\documentclass[12pt]{article}

\usepackage[section]{placeins}
\usepackage{natbib}
\usepackage[letterpaper, margin=1in]{geometry}

\usepackage[todo,table]{MathEnv}
\usepackage[stat,commdiag,alg]{MathShorthand}

\usepackage{breakcites}

\usepackage{tikz}
\usetikzlibrary{3d,calc}

\newcommand{\indeg}{d^\text{in}}
\newcommand{\outdeg}{d^\text{out}}
\newcommand{\new}[1]{{\color{blue} #1}}

\title{The Asymptotics of the Expected Betti Numbers of Preferential Attachment Clique Complexes}
\author{Chunyin Siu, Gennady Samorodnitsky, Christina Lee Yu and Rongyi He}
\date{}

\begin{document}

\maketitle

\begin{abstract}
The preferential attachment model is a natural and popular random graph model for a growing network that contains very well-connected ``hubs''. We study the higher-order connectivity of such a network by investigating the topological properties of its clique complex. We concentrate on the expected Betti numbers, a sequence of topological invariants of the complex related to the numbers of holes of different dimensions. We determine the asymptotic growth rates of the expected Betti numbers, and prove that the expected Betti number at dimension 1 grows linearly fast, while those at higher dimensions grow sublinearly fast. Our theoretical results are illustrated by simulations. \new{(Changes are made in this version to generalize Proposition 14 and to streamline proofs. These changes are shown in blue.)}
\end{abstract}

\section{Introduction}

From biology to the internet, real-world networks in a wide range of fields are believed to be \emph{scale-free}, in the sense that the degree distributions of these graphs obey certain power laws, often with infinite variance. \citep{albert05_cellNetworks_sclaeFree,eguluz05_brain_scaleFree,faloutsos99_internet_powerLaw,mislove07_onlineSocialNetworks_scalefree,gjoka10_facebook_network,caldarelli07_scalefree} In addition to the degree distributions of these networks, there has been significant interest in their higher-order connectivity, defined in various ways. \citep{benson16_complexNetwork_motif,nolte20_brain_higherOrderNetwork,lambiotte19_higherOrderNetworks,battiston20_higherOrderNetworks,bianconi21_higherOrderNetwork} Some of these notions are graph-theoretic, like clustering coefficients \citep{watts98_clusteringCoefficient} and motif counts \citep{benson16_complexNetwork_motif}. One may also understand higher-order connectivity by viewing the network as a higher-dimensional analogue of a graph, like a \emph{hypergraph} or a \emph{simplicial complex}. In particular, the \emph{Betti numbers} of simplicial complexes, a concept from algebraic topology, generalize the counts of connected components and cycles to the counts of higher-dimensional holes, and they have proven to be useful statistics in recent topological-data-analytic applications. \citep{aktas19_TDA_survey_network,carlsson09_topodata,chazal21_TDA_survey_data} 
For instance, it was observed that the way holes emerge in (biological) neural networks helps distinguish different stimuli to the brains of different animals. \citep{reimann17_neuralCliques} However, understanding Betti numbers is in general difficult for conceptual, computational and analytical reasons. To the best of our knowledge, there has been no analytical results on the Betti numbers of scale-free simplicial complexes, with the notable exception of \citep{oh21_bettiNumbers_trianglePreferentialAttachment_bettiNumber}, which considers a 2-dimensional model and determines the asymptotics of the Betti number at dimension 1.

\begin{figure}[t]
\centering
\includegraphics[height = 5.5cm]{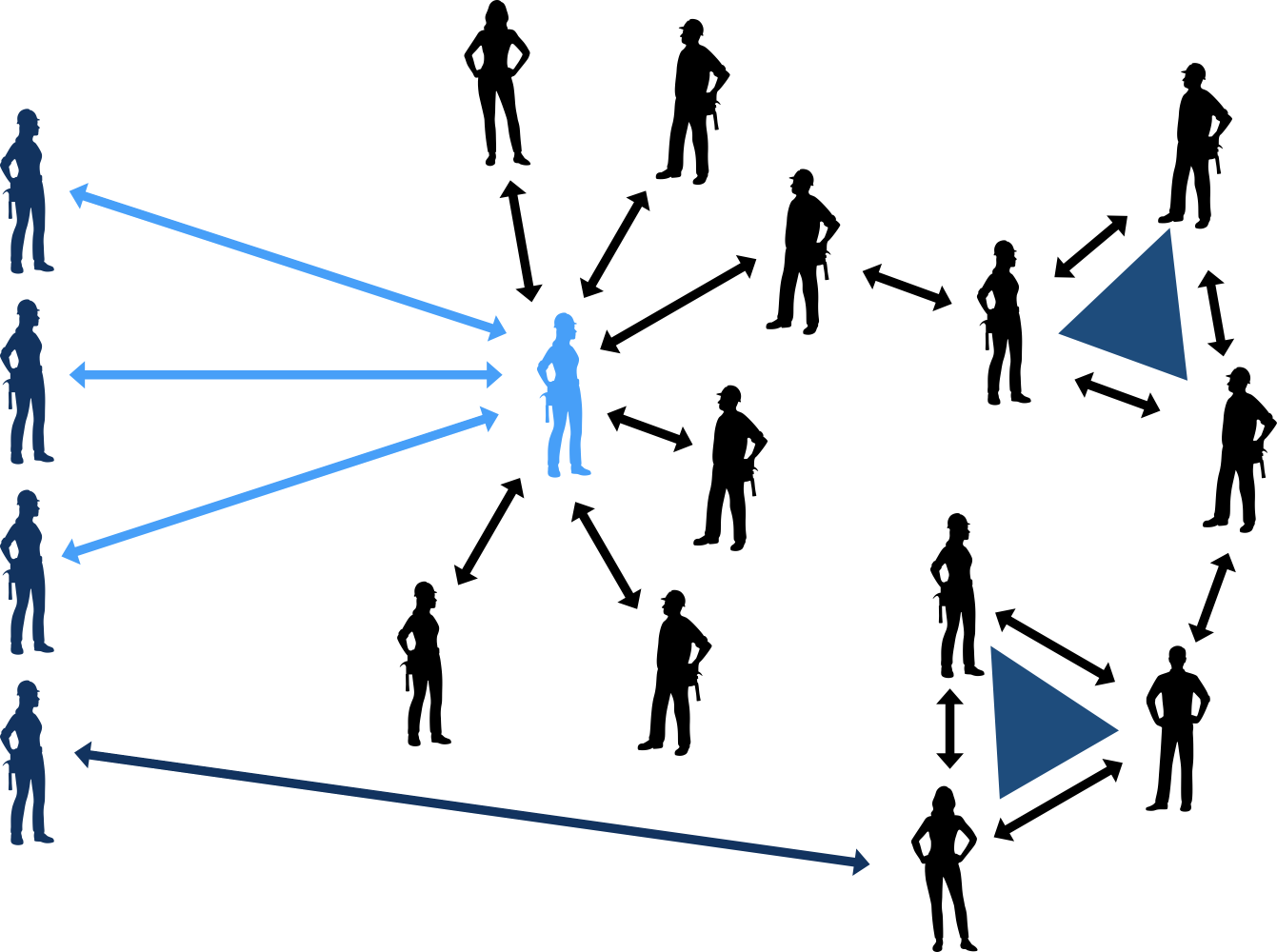}
\includegraphics[height = 5.5cm]{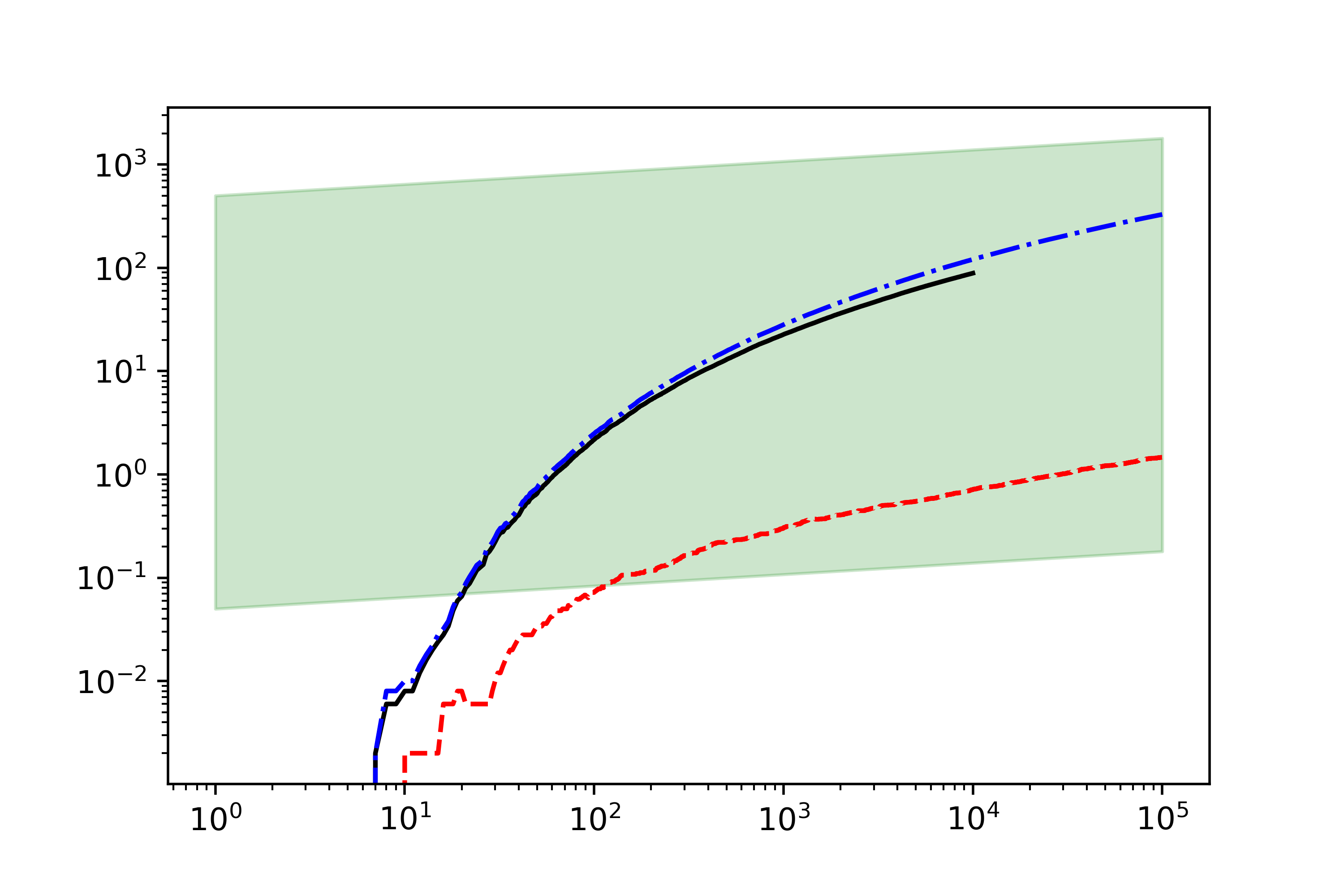}
\caption{(Left) An illustration of the preferential attachment and clique building mechanism. When new nodes (drawn as people) in the left column are added to the network, they are more likely to attach to already popular nodes (who have high degrees), like the light blue person in the figure. Fully connected subsets of nodes form triangles, tetrahedra and their higher-dimensional analogues in the clique complex. Note that in order to have triangles, each new node must connect to at least 2 nodes, but we only draw one connection for each new node to keep the illustration simple.
(Right) The log-log plot of the evolution of the mean Betti number at dimension 2 for 500 preferential attachment clique complexes. The horizontal axis is the number of nodes in log-scale, the black curve corresponds to the mean Betti number, in log scale as well. The dotted curves correspond to the mean upper and lower bounds in our argument (specifically in \cref{prop:decomposition}). The slope of the shaded region is the asymptotical growth rate of the expected Betti number. The position and the width of the shaded region is chosen posthoc manually because the theoretical constants are too conservative.}
\label{fig:teaser}
\end{figure}

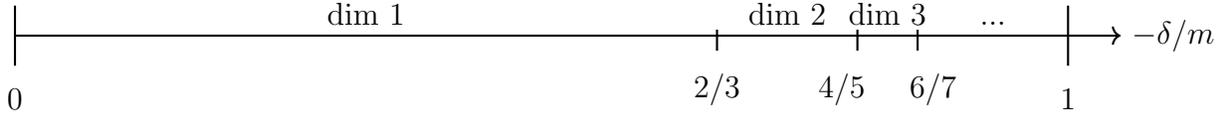
\begin{figure}[t]
\centering

\begin{tikzpicture}[x={(0cm, 8cm)}, y={(28cm, 0cm)}, thick, scale=0.5]


  \draw (-0.1,0) node[label=below:$0$] {} -- (0.1,0);
  \draw (-0.05,2/3) node[label=below:$2/3$] {} -- (0.02,2/3);
  \draw (-0.05,4/5) node[label=below:$4/5 \quad$] {} -- (0.02,4/5);
  \draw (-0.05,6/7) node[label=below:$\quad 6/7$] {} -- (0.02,6/7);
  \draw (-0.1,1) node[label=below:$1$] {} -- (0.1,1);
  
  \draw (0,0) -- (0,2/3) node[midway, above] {dim 1};
  \draw (0,2/3) -- (0,4/5) node[midway, above] {dim 2};
  \draw (0,4/5) -- (0,6/7) node[midway, above] {dim 3};
  \draw (0,6/7) -- (0,1) node[midway, above] {...};
  \draw[->] (0,1) -- (0,1.05) node[right]{$-\delta/m$};
\end{tikzpicture}
\caption{The top dimensions with unbounded expected Betti numbers for different values of $-\delta/m \in [0, 1)$ for $m$ not too small (recall that $-\delta/m$ increases with the strength of preferential attachment effect; see \cref{thm:expected_BettiNumbers} for the precise condition on $m$). $2/3$, $4/5$ and $6/7$ are the critical thresholds for dimensions $2$, $3$ and $4$ respectively.}
\label{fig:phase_diagram}
\end{figure}

In this paper, we use \emph{preferential attachment graphs} as models of scale-free networks, and we analytically determine the asymptotics of the expected Betti numbers of their \emph{clique complexes} (also known as \emph{flag complexes}) at all dimensions. An illustration of the clique complex of a preferential attachment graph is shown on the left panel of \cref{fig:teaser}. We show that, under mild assumptions, the expected Betti number at dimension 1 is asymptotically linear in the number of nodes, and those at higher dimensions are bounded between constant multiples of sublinear powers of the number of nodes. The precise statements of our main results, \cref{thm:expected_BettiNumbers} and \cref{prop:trivial_cases}, are given at the end of \cref{sec:main_results}.

Our results suggest that preferential attachment favors higher-order connectivity. For each dimension larger than $1$, when the preferential attachment effect reaches a critical threshold, the random simplicial complex undergoes a phase transition, in the sense that the expected Betti number at that dimension ceases to be bounded, and diverges to infinity as the number of nodes increases. Therefore, when the preferential attachment effect increases, more and more topological phenomena become possible; and the topological complexity at each dimension increases as well. The critical thresholds for the lowest dimensions are illustrated in \cref{fig:phase_diagram}.

On the right panel of \cref{fig:teaser}, we illustrate the sublinear growth of the average Betti number at dimension 2 through a simulation. Our theorem and our choice of parameters dictate that the curve for the evolution of the expected Betti number is eventually contained in a band with slope $2/9$, which is the slope of the shaded region. The evolution of our estimate of the expected Betti number is plotted as the solid curve, and it is plausible that it remains inside the shaded region when extrapolated indefinitely. We discuss the simulation in greater detail in \cref{sec:numerical_simulation}.

\paragraph{Scale-Free Networks} Scale-free networks have been widely studied in the network analysis community. In \citep{barabasi99_preferentialAttachment}, Barabasi and Albert proposed modeling scale-free networks with preferential attachment graphs. Since then the model has been refined and generalized in many different ways \citep{bollobas01_preferentialAttachemnt_LCD,buckley04_preferentialAttachment,holme02_preferentialAttachment}. Regarding higher-order connectivity of preferential attachment graphs, their \emph{clustering coefficients} has been well studied (see, for instance, \citep{bollobas02_scaleFree,holme02_preferentialAttachment,eggemann11_preferentialAttachment_clusteringCoefficient,ostroumova13_preferentialAttachment_clusteringCoefficient,prokhorenkova17_preferentialAttachment_clusteringCoeff}). In \citep{garavaglia19_subgraphs_preferentialAttachment}, the growth rate of the expected number of small \emph{motifs} in preferential attachment graphs was determined. In \citep{bianconi16_growingSimplicialComplex,courtney17_growingWeightedSimplicialComplex}, some forms of preferential attachment simplicial complexes, were considered, and power laws for some forms of higher-dimensional degrees were found. However, little is known about topological invariants of scale-free simplicial complexes beyond dimension 1 \citep{oh21_bettiNumbers_trianglePreferentialAttachment_bettiNumber}.

\paragraph{Random Simplicial Complexes} In general, the literature on \emph{random simplicial complexes} has been growing rapidly in recent years. \citep{kahle14_randomSimplicialComplexes_survey,bobrowski22_randomSimplicialComplexes_survey} Of particular interest to us are random clique complexes, and we review some results about these complexes below.

Being natural extensions of Erdos-Renyi graphs, Erdos-Renyi clique complexes have been studied intensively. The Betti numbers of Erdos-Renyi clique complexes were found to be supported at a critical dimension with high probability in \citep{kahle09_randomCliqueComplex,kahle14_randomCliqueComplex}. This result can be seen as a generalization of the classical result on the phase transitions of Erdos-Renyi graphs \citep{erdos59_erdosRenyi}. Limit theorems for the Betti number at the critical dimension were established \citep{kahle13_randomComplexes_betti_limitTheorems}. Multivariate central limit theorems for certain counts for these complexes were established in \citep{temcinas22_edrosRenyi_clique_multivariateCLT}. The fundamental groups of these complexes were studied in \citep{costa15_erdosRenyi_clique_fundamenalGroup}. These complexes were shown to be homotopy equivalent to a simplicial complex of the critical dimension with high probability in \citep{malen23_erdosRenyi_clique_collapsible}. Persistence homology for Erdos-Renyi clique complexes was considered in \citep{ababneh22_erdosRenyi_clique_persistence}.

A clique complex of a graph may be regarded as a (radius-1) \emph{Rips complex} with respect to the graph metric. Rips complexes are geometric clique complexes where nearby points are connected. They have gained substantial attention due to their applications in topological data analysis \citep{bobrowski18_randomGeometricComplexes_survey}. Limit theorems for the Betti numbers of random Rips complexes constructed from independent and identically distributed points were established in \citep{kahle11_geometricComplex,kahle13_randomComplexes_betti_limitTheorems}. They are generalizations of connectivity results for random geometric graphs in \citep{penrose2003_randomGeometricGraphs}. The asymptotics of the expected Betti numbers of the Rips complexes of stationary point processes were established in \citep{yogeshwaran15_geometricComplexes_stationaryPP_betti}. The persistent homology of Rips complexes was studied in \citep{bobrowski17_randomgometricComplexes_persistentCycles,hiraoka18_persistenceDiagrams_limitTheorems}, among many other works. Functional limit theorems for the Euler characteristics of Rips complexes were established in \citep{thomas21_randomGeometricComplexes,thomas21_extrememSamplePointCloud_eulerCharacteristics_limitTheorem,owada22_topologicalCrackle_functionalSLLN}.

\paragraph{Comparison with Other Random Simplicial Complexes} Our results suggest the following similarities between preferential attachment clique complexes, Erdos-Renyi clique complexes \citep{kahle13_randomComplexes_betti_limitTheorems} and random Rips complexes \citep{kahle11_geometricComplex,yogeshwaran15_geometricComplexes_stationaryPP_betti}:
\begin{itemize}
\item the major contribution to the Betti numbers are due to small cycles; and
\item the expected Betti numbers have a dominating dimension, which is 1 in this case.
\end{itemize}
On the other hand, our complexes have two distinctive features.
\begin{itemize}
\item Unlike Erdos-Renyi clique complexes, there is a range of dimensions with positive expected Betti numbers, which grow with the number of nodes.
\item Compelled by the preferential attachment mechanism, holes are predominantly represented by highly inter-connected cycles.
\end{itemize}

\paragraph{Dominating Cycles and Proof Synopsis}

We describe these small inter-connected cycles of dimension $2$ in the last bullet point above. The case for higher dimensions is similar. For each positive integer $k$, let $\Gamma_k$ be the graph consisting of 4 vertices forming a square (with no diagonals) and $k$ other vertices that are each connected to all corners of the square. $\Gamma_3$ is shown in \cref{fig:Gamma3}. At dimension $2$, the Betti number of the clique complex of this graph is $k-1$, which is one less than the number of distinct (but not disjoint) copies of $\Gamma_1$ in $\Gamma_k$ when the two graphs are suitably directed, as in \cref{fig:Gamma3}. Therefore, the contribution of such subgraphs to the Betti number can be approximated by the number of $\Gamma_1$'s in the graph.

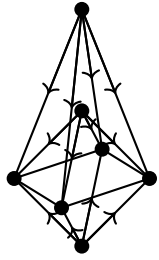
\begin{figure}[t]
\centering
\begin{tikzpicture}[x={(2cm, 0cm)}, y={(0.6cm, 0.866cm)},z={(0cm,2cm)}, thick, scale=0.45]
  \coordinate (A) at (1,0,0);
  \coordinate (B) at (0,1,0);
  \coordinate (C) at (-1,0,0);
  \coordinate (D) at (0,-1,0);
  \coordinate (E) at (0,0,-1);
  \coordinate (F) at (0,0,1);
  \coordinate (G) at (0,0,2.5);

  \draw (A) node[circle,fill,inner sep=2pt] {};
  \draw (B) node[circle,fill,inner sep=2pt] {};
  \draw (C) node[circle,fill,inner sep=2pt] {};
  \draw (D) node[circle,fill,inner sep=2pt] {};
  \draw (E) node[circle,fill,inner sep=2pt] {};
  \draw (F) node[circle,fill,inner sep=2pt] {};
  \draw (G) node[circle,fill,inner sep=2pt] {};
  
  \draw (A) -- (B) -- (C) -- (D) -- cycle;
  \foreach \x in {A,B,C,D}
    \foreach \y in {E,F,G}{
      \draw (\y) -- (\x);
      \draw[->] (\y) -- ($(\y)!0.5!(\x)$);
      }

\end{tikzpicture}
\caption{The graph $\Gamma_3$. Edges point towards the horizontal square. Unmarked edges can be directed arbitrarily.}
\label{fig:Gamma3}
\end{figure}

When we approximate the Betti number of the whole complex with the number of copies of $\Gamma_1$'s, the error term of the approximation consists of a few parts. First, other subcomplexes may also add to the Betti number, like (the clique complexes of) subgraphs with vertices attached to pentagons rather than squares. However, such subcomplexes are more complicated, and hence fewer of them arise from the preferential attachment mechanism.

On the other hand, copies of $\Gamma_k$ may not represent actual holes because they could be boundaries of something else. (e.g. the boundary of a solid ball is a hollow sphere but the ball itself is not hollow.) However, again, such ``something else" objects are also more complicated and fewer of them arise. For technical reasons, we will analyze two types of boundaries separately.

We justify this description in \cref{sec:main_idea} and the proof of \cref{lem:link_BettiNumber}. \cref{lem:prob_link_contains_sphere,lem:link_BettiNumber} together give the asymptotics of the main term. The error terms are handled by \cref{lem:miscarriage} and another application of \cref{lem:link_BettiNumber}.

\paragraph{Proof Techniques}
Like many other works on random simplicial complexes, we rely on subgraph counting results and a characterization of minimal cycles. Specifically, we rely on Garavaglia and Stegehuis' subgraph counting results in \citep{garavaglia19_subgraphs_preferentialAttachment}, and we generalize a result on minimal cycles due to Gal and Kahle \citep{gal05_clique_spheres,kahle09_randomCliqueComplex} to the setting of relative homology using an exact sequence argument. The use of relative homology is needed for the lower bound on the expected Betti numbers. More details can be found in \cref{sec:main_idea}.

As in \citep{kahle14_randomCliqueComplex}, \emph{links} play a crucial role in our argument to simplify the complicated connections of our dominating minimal cycles. Unlike in \citep{kahle14_randomCliqueComplex}, our approach (\cref{prop:decomposition}) is based on homological algebra. Garland's spectral approach \citep{garland73_vanishingCohomology} is not applicable because we know the Betti numbers do not vanish.

Establishing lower bounds on Betti numbers is often difficult. The strong Morse inequality (Theorem 1.8 of \citep{forman02_discreteMorse_guide}) is not applicable beyond the dominating dimension, which is 1, and we are not aware of applicable techniques for our case in the literature. In \citep{kahle14_randomCliqueComplex}, an unfilled hole is constructed to give the lower bound of $1$. In \citep{kahle11_geometricComplex,yogeshwaran15_geometricComplexes_stationaryPP_betti}, the Betti numbers are \emph{equal} to the counts of certain subcomplexes in the sparse regime, because distinct cycles are not connected. This is not applicable in our case because our cycles are interconnected.

We obtain our lower bound by finding some copies of the $\Gamma_k$'s, and subtracting, from the counts of these copies, the sum of probabilities that the $\Gamma_1$'s in these $\Gamma_k$'s fail to contribute to the Betti number. A relative-homology argument relates this sum to another count of certain subcomplexes, which can estimated with a probabilistic argument.

We remark that our method gives estimates of the expected \emph{persistent Betti numbers} as well for the filtration of node arrival (the complex at time $t$ consists of the first $t$ nodes), but to keep the exposition simple we do not discuss this point in the present paper.

\paragraph{Paper Outline}

The rest of the paper is organized as follows. We give basic definitions and state the main results in \cref{sec:main_results}. We first review the background materials in \cref{sec:background}. We begin proving the main theorem, \cref{thm:expected_BettiNumbers}, in \cref{sec:main_idea} by presenting a decomposition result, which justifies our discussion on copies of $\Gamma_1$'s above. We complete the proof of \cref{thm:expected_BettiNumbers} in \cref{sec:proof_main_result}. In \cref{sec:proof_trivial_cases}, we handle the boundary cases left out by the theorem by proving \cref{prop:trivial_cases}. We present and discuss the simulation results in \cref{sec:numerical_simulation}, and we discuss future directions in \cref{sec:future_directions}.

\paragraph{Acknowledgement} 

This research was partially supported by the AFOSR grant FA9550-22-1-0091 and by the Cornell University Center for Advanced Computing, which receives funding from Cornell University, the National Science Foundation, and members of its Partner Program.

The authors would like to thank Shmuel Weinberger, Jason Manning, Takashi Owada, Andrew Thomas and Benjamin Thompson for insightful discussion.

\section{Basic Definitions and Main Results}
\label{sec:main_results}

Our main results concern the expected Betti numbers of preferential attachment clique complexes. We will state them precisely after introducing all relevant terms and notations.

\paragraph{Preferential Attachment Graphs} We first define preferential attachment graphs. Incomplete lists of the multitude of variants of such graphs can be found in \citep{prokhorenkova17_preferentialAttachment_clusteringCoeff} and Chapter 4.3 of \citep{garavaglia19_preferentialAttachment_thesis}. We adopt the affine model in \citep{garavaglia19_subgraphs_preferentialAttachment}, since we will rely heavily on the subgraph count results therein.

\begin{definition}[(Affine) Preferential Attachment Graphs; Definition 1 of \citep{garavaglia19_subgraphs_preferentialAttachment}]
\label{def:preferential_attachment}
Let $T, m$ be positive integers with $T \geq 2$ and let $\delta \in (-m, \infty)$. The preferential attachment graph $G(T, \delta, m)$ is the random directed graph with no self-loops but possibly with repeated edges that is constructed inductively as follows.
\begin{itemize}
\item $G(2, \delta, m)$ is the deterministic graph with two nodes, indexed by 1 and 2, and $m$ edges from node 2 to node 1.
\item $G(T, \delta, m)$ is constructed by adding a node, with index $T$, to $G(T - 1, \delta, m)$ and $m$ edges from node $T$ to randomly chosen nodes in $G(T - 1, m, \delta)$. For $1 \leq v \leq T-1$, the probability that node $v$ is chosen is proportional to $\delta$ plus the degree of node $v$, which is updated every time an edge is added.
\end{itemize}
\end{definition}

We will consider the preferential attachment graphs $G(T, \delta, m)$ with $\delta < 0$. The more negative $\delta$ is, the stronger the preferential attachment effect.

\paragraph{Simplicial Complexes, Homology and Betti Numbers} Next, we introduce simplicial complexes and Betti numbers, which are the \emph{ranks} of \emph{homology groups}. The precise definitions of these topological concepts are long and technical. We will describe these concepts somewhat informally and point the reader to the literature for the precise definitions. We refer the reader to \citep{giblin_2010_homology} for an elementary introduction, Chapter 2 of \citep{hatcher02_algtopo} for a thorough exposition, and \citep{munkres84algtopo} for a rigorous exposition based on simplicial complexes.


A simplicial complex is a combinatorial object formed by gluing together vertices, edges, triangles, tetrahedra and higher-dimensional \emph{simplices} in a principled manner. The precise definition can be found in Definitions 3.14 and 3.18 of \citep{giblin_2010_homology}, and Chapters 1.2 and 1.3 of \citep{munkres84algtopo}. 
A simplicial complex $X$ is said to be a \emph{clique complex} if it contains a simplex whenever it contains all edges of the simplex. For example, the filled triangle is a clique complex, while the hollow triangle is not. The clique complex of a simple graph is the simplicial complex formed by adding a simplex to the graph whenever all edges of the simplex are in the graph.

The homology group $H_q(X)$ of a simplicial complex $X$ at dimension $q$ is an abelian group that consists of equivalence classes of $q$-dimensional \emph{cycles}. Two cycles that differ by the \emph{boundary} of a formal sum of $(q+1)$-dimensional simplices are considered equivalent. These equivalence classes, called \emph{homology classes}, are formal sums of $q$-dimensional ``holes" in the simplicial complex, with the exception that $H_0(X)$ consists of formal sums of path components. For example, for the $n$-dimensional sphere $S^n$ and the $n$-dimensional solid ball $D^n$ (suitably triangulated), 
\begin{equation}\label{eqn:sphere_betti}
H_q(S^n) = \begin{cases}
\mathbb{Z} & \text{ if } q \in \{0, n\}\\
0 & \text{ otherwise; }
\end{cases}
\qquad
H_q(D^n) = \begin{cases}
\mathbb{Z} & \text{ if } n = 0\\
0 & \text{ otherwise.}
\end{cases}
\end{equation}
The equivalence relation of cycles weeds out cycles that witness ``fake" holes. For example, the boundary sphere of a solid ball is a cycle, but it does not really represent a hole in the solid ball. 
The precise definition of homology groups can be found in Definition 4.8 of \citep{giblin_2010_homology}, Chapter 1.5 of \citep{munkres84algtopo} and the section titled ``Simplicial Homology" in Chapter 2.1 of \citep{hatcher02_algtopo}.


The Betti number $\beta_q(X)$ of a simplicial complex $X$ at dimension $q$ is the rank of $H_q(X)$. In symbols, $\beta_q(X) = \rk H_q(X)$. The rank of a group is the size of a maximal linearly independent subset. If the group is a vector space, then the rank is the dimension. Roughly speaking, the Betti numbers are the numbers of components and holes in $X$. For instance, 
\begin{equation*}\label{eqn:sphere_betti}
\beta_q(S^n) = \begin{cases}
1 & \text{ if } q \in \{0, n\}\\
0 & \text{ otherwise.}
\end{cases}
\end{equation*}
The precise definition of Betti numbers can be found on p.155 of \citep{giblin_2010_homology}, p.24 of \citep{munkres84algtopo} and p.130 of \citep{hatcher02_algtopo}.

\paragraph{Main Results} We now state our main results. We adopt the asymptotical notations in \cref{tab:asymptotics}, where $f$ and $g$ are assumed to be nonnegative functions defined on $\mathbb{N}$. Following the notation in \citep{garavaglia19_subgraphs_preferentialAttachment}, let
\begin{equation}\label{eqn:chi}
\chi(\delta, m) = 1 - \frac{1}{2 + \delta/m} \in (0, 1/2).
\end{equation}
When $\delta$ becomes more negative, $\chi$ decreases, and hence $\chi$ decreases with strength of the preferential attachment effect. It is important not to confuse $\chi$ with the Euler characteristic, which is also commonly denoted by $\chi$. 

We define the preferential attachment clique complex $X(T, \delta, m)$ as follows.

\begin{definition}[Preferential Attachment Clique Complex]
Let $G(T, \delta, m)$ be the preferential attachment graph, and $G_\text{simple}(T, \delta, m)$ be the graph obtained by replacing all repeated edges in $G(T, \delta, m)$ with simple edges. The preferential attachment clique complex $X(T, \delta, m)$ is the clique complex of $G_\text{simple}(T, \delta, m)$.
\end{definition}

Our main result is as follows.

\begin{theorem}\label{thm:expected_BettiNumbers}
Let $X(T, \delta, m)$ be the preferential attachment clique complex. Let $q \geq 2$ and suppose $m \geq 2q$. Then
\begin{equation*}
E(\beta_q(X(T, \delta, m)) = 
\begin{cases}
\Theta(T^{1 - 2q\chi(\delta, m)}) & \text{ if } 1 - 2q\chi(\delta, m) > 0 \\
\Theta(\log T) & \text{ if } 1 - 2q\chi(\delta, m) = 0 \\
O(1) & \text{ otherwise,}
\end{cases}
\end{equation*}
where the big-Oh and big-$\Theta$ constants (see \cref{tab:asymptotics}) depend only on $q, \delta, m$, and $E$ denotes expectation.
\end{theorem}

One could keep track of the big-$\Theta$ constants in the proof. For instance, when $m = 7, \delta = -5, q = 2$, the big-$\Theta$ constants can be chosen to be $C = 2.16 \times 10^{14}$ and $c = 1.18 \times 10^{-34}$. Finding the optimal constants, however, is beyond the scope of this work.

The theorem leaves out the trivial cases for $q < 2$ or $m < 2q$, which we address below.

\begin{proposition}\label{prop:trivial_cases}
Let $X(T, \delta, m)$ be the preferential attachment clique complex. Then
\begin{itemize}
\item For $q = 0$, $\beta_0(X(T, \delta, m)) \equiv 1$.
\item For $q = 1$, 
$
E[\beta_1(X(T, \delta, m))] = (m-1) T + o(T)
$
\item For $m < 2q$, $\beta_q(X(T, \delta, m)) \equiv 0$.
\end{itemize}
\end{proposition}

\begin{remark}
A notion of clique complex for multigraphs is defined in \citep{ayzenberg20_cliqueMultigraph}. This complex is not a simplicial complex but it is a regular CW complex. Simplices in this complex are fully connected subsets with labelled edges. We do not use this notion to keep the exposition simple; this notion also presents challenges to numerical simulations. We remark that our argument gives the same bound for this notion of clique complex upon slight modifications, because most edges in preferential attachment graphs are simple, and repeated edges only change expected subgraph counts by at most a constant factor.
\end{remark}

\begin{table}
\centering
\caption{Asymptotical notations}
\label{tab:asymptotics}
\begin{tabularx}{.8\textwidth}{|c|X|}
\hline
Notation & Definition \\
\hline\hline
$f(n) = O(g(n))$ & $f(n) \leq Cg(n)$ for large enough $n$ for some positive constant factor $C$. \\ \hline
$f(n) = \Omega(g(n))$ & $f(n) \geq cg(n)$ for large enough $n$ for some positive constant factor $c$. \\ \hline
$f(n) = \Theta(g(n))$ & $cg(n) \leq f(n) \leq Cg(n)$ for large enough $n$ for some positive constant factors $C$ and $c$. \\ \hline
$f(n) = o(g(n))$ & $\lim_{n \to \infty} f(n)/g(n) = 0$
\\\hline
\end{tabularx}
\end{table}

\section{Preliminaries}
\label{sec:background}

We recall relevant results in the literature in this section. First, we recall subgraph counting results in \cref{sec:bg_preferentialAttachment}. Then, in \cref{sec:bg_homologicalAlgebra}, we recall facts from homological algebra that helps us manipulate homology groups. We define a few simplicial complexes and subcomplexes that will be featured in our proof in \cref{sec:topo_notations}. Finally, we conclude this section by stating our generalization of Gal and Kahle's result on minimal clique cycles in \cref{sec:topological_facts}. We defer its homological-algebraic proof to \cref{sec:proof_minimallyNontrivialSubcomplex_1skeleton}. 

\subsection{Preferential Attachment Graphs}
\label{sec:bg_preferentialAttachment}

Recall that preferential attachment graphs are defined in \cref{def:preferential_attachment}, and we need certain subgraph count results, namely \cref{thm:subgraphsCounts_preferentialAttachment,prop:subgraphsProbs_preferentialAttachment}. We develop a few definitions to simplify their statements.

The preferential attachment graph $G(T, \delta, m)$ is a random subgraph of the underlying \emph{attachment graph} $U(T, m)$, which we define as follows.

\begin{definition}[Attachment Graph]
The $(T, m)$-attachment graph $U(T, m)$ is the directed multigraph such that
\begin{itemize}
\item the vertex set is $\{1, ..., T\}$,
\item there are $m$ edges between any pair of distinct nodes, and 
\item all edges point from later nodes to earlier nodes.
\end{itemize}
\end{definition}


The following definition will simplify the expression of the estimate in the theorem.

\begin{definition}[Degree and Power]
Let $\Gamma$ be a subgraph of $U(T, m)$ and $v$ be a vertex in $\Gamma$. Denote by $\indeg_\Gamma(v)$ and $\outdeg_\Gamma(v)$ the in- and out-degrees of $v$ with respect to $\Gamma$. The power $p_\Gamma(v)$ of $v$ is
\begin{align*}
p_\Gamma(v) &= - \left[\outdeg_\Gamma(v) + (\indeg_\Gamma(v) - \outdeg_\Gamma(v))(1 - \chi(\delta, m)) \right]
\\&= - \left[ (1 - \chi(\delta, m)) \indeg_H(v) + \chi(\delta, m)\outdeg_H(v)\right],
\end{align*}
where $\chi(\delta, m)$ is defined in \cref{eqn:chi}.
\end{definition}


\begin{theorem}[Theorem 1 of \citep{garavaglia19_subgraphs_preferentialAttachment}]
\label{thm:subgraphsCounts_preferentialAttachment}
Consider the preferential attachment graph $G(T, \delta, m)$.
Let $\Gamma$ be a subgraph of $U(T, m)$ with vertex set $V_\Gamma = \{v_1 < ... < v_{|V_\Gamma|}\}$ and with out-degrees bounded above by $m$. Then 
the expected count of subgraphs in $G(T, \delta, m)$ that are isomorphic to $\Gamma$ is
$$\Theta(T^A (\log T)^{r-1}),$$
where $A$ is the maximum value of the sequence $a_0, ..., a_{|V_\Gamma|}$ defined by
\begin{equation}\label{eqn:PAM_count_sequence}
a_k = |V_\Gamma| - k + \sum_{l > k} p_\Gamma(v_l),
\end{equation}
and $r$ is the number of maximizers. The big-$\Theta$ constants are independent of $T$ but do depend on $m$, $\delta$ and the in- and out-degree sequences of $\Gamma$.
\end{theorem}

\begin{remark}
The original theorem is stated in terms of $\tau = 3 + \delta/m$ (see the end of Section 2.1 of \citep{garavaglia19_subgraphs_preferentialAttachment}) rather than $\delta, m$ or $\chi(\delta, m)$ . We do not use $\tau$ because nowhere is $\tau$ used in the proof and using $\tau$ does not simplify our expressions.
\end{remark}


\cref{thm:subgraphsCounts_preferentialAttachment} can be proven from the following lemma, which we will need for the lower bound of the expected Betti numbers.


\begin{proposition}[Lemma 1 of \citep{garavaglia19_subgraphs_preferentialAttachment}]
\label{prop:subgraphsProbs_preferentialAttachment}
Let $\Gamma$ be a subgraph of $U(T, m)$ with vertex set $V_\Gamma \subseteq \{1, ..., T\}$ and with out-degrees bounded above by $m$. Then the probability that $G(T, \delta, m)$ contains $\Gamma$ is
$$\Theta \left(\prod_{v \in V_\Gamma} v^{p_\Gamma(v)} \right),$$
where the big-$\Theta$ constants depend only on $\delta$, $m$, and the in-degree sequence of $\Gamma$.
\end{proposition}

\subsection{Homological Algebra}
\label{sec:bg_homologicalAlgebra}

The technical definition of homology offers a host of homological-algebraic tools for deducing relationships between homology groups of different simplicial complexes, which are encoded by homomorphisms induced by maps between the complexes. The most important tool is called \emph{exact sequences}, which are generalizations of the rank-nullity theorem. Specifically, we will need the \emph{long exact sequence for triplets}, and the \emph{Mayer-Vietoris sequence}. We state them at the end of this section.

Every map $f: X \to Y$ between two simplicial complexes induces a homomorphism $f_q: H_q(X) \to H_q(Y)$ at every dimension $q$. This homomorphism maps a cycle in $X$ to the image of the cycle under $f$. Interesting relationships between the homology of a complex $Y$ and a subcomplex $X \subseteq Y$ are often revealed by the homomorphism induced by the inclusion map $f: X \to Y$ (defined by $f(x) = x$).

An exact sequence is a sequence of groups $(A_n)$ with homomorphisms $\varphi_n: A_n \to A_{n+1}$ between adjacent groups, such that $\ker \varphi_n = \im \varphi_{n-1}$. The sequence is said to be \emph{short} if it starts and ends with the trivial group $0$. We will use the following fact repeatedly.

\begin{proposition}[Alternating Sum of Ranks of a Short Exact Sequence; Exercise 3.16 of \citep{rotman2008_homologicalAlgebra}]
Let $0 \to A_1 \to ... \to A_n \to 0$ be a short exact sequence of finitely generated abelian groups. Then
$$\sum_i (-1)^k \rk A_k = 0, \text{ and }$$
$$\rk A_k \leq \rk A_{k-1} + \rk A_{k+1} \text{ for every }1 < k < n.$$
\end{proposition}

The rank-nullity theorem is a special case of the first fact for $n = 3$.

\begin{proof}
The first claim is Exercise 3.16 of \citep{rotman2008_homologicalAlgebra}. For the second claim, 
note that the following sequence is also exact:
$$0 \to \im \varphi_{k-1} \xrightarrow{i} A_k \xrightarrow{\varphi_k} \im \varphi_k \to 0,$$
where $i$ denotes the inclusion homomorphism. The first claim in the proposition implies
$$\rk A_k = \rk \im \varphi_{k-1} + \rk \im \varphi_k \leq \rk A_{k-1} + \rk A_{k+1}.$$
\end{proof}

To state the long exact sequence for triplets, we need to define \emph{pairs} and \emph{triplets} of simplicial complexes. A 2-tuple $(X, A)$ of simplicial complexes is said to be a pair if $A$ is a subcomplex of $X$, and a 3-tuple $(Y, X, A)$ is said to be a triplet if $(Y, X)$ and $(X, A)$ are two pairs. A simplicial complex $X$ can be identified with the pair $(X, \emptyset)$. The homology group $H_q(X, A)$ of a pair $(X, A)$ at dimension $q$ is called the \emph{relative homology group} of $X$ at dimension $q$ relative to $A$. Its rank is called the (relative) Betti number of the pair and is denoted by $\beta_q(X, A)$.

In homology theory, pairs play the role that quotient vector spaces play in linear algebra. $(X, A)$ acts like the ``quotient" of $X$ by $A$, and $(Y, X)$ acts like the ``quotient" of $(Y, A)$ by $(X, A)$. Precisely, we have the following proposition.

\begin{proposition}[Long Exact Sequence for Triplets; Theorem 7.8 of \citep{giblin_2010_homology}, Exercise 24.1 of \citep{munkres84algtopo}, p.118 of \citep{hatcher02_algtopo}]
For every triplet $(Y, X, A)$, there exists a homomorphism $\partial_q: H_q(Y, X) \to H_{q-1}(X, A)$ for each $q$ such that following sequence is exact
$$... \to H_q(X, A) \to H_q(Y, A) \to H_q(Y, X) \xrightarrow{\partial_q} H_{q-1}(X, A) \to H_{q-1}(Y, A) \to ...,$$
where all unmarked maps are induced by inclusions.
\end{proposition}
 
The Mayer-Vietoris sequence is useful for divide-and-conquer arguments, as it relates the homology groups of a space with those of its subspaces.

\begin{proposition}[Mayer-Vietoris Sequence; Theorem 7.2 of \citep{giblin_2010_homology}; Theorem 25.1 of \citep{munkres84algtopo}, p.149 of \citep{hatcher02_algtopo}]
Let $X$ and $Y$ be subcomplexes of a simplicial complex $Z$. Then there exist maps such that the following sequence is exact.
$$... \to H_q(X \cap Y) \to H_q(X) \oplus H_q(Y) \to H_q(X \cup Y) \to H_{q-1}(X \cap Y) \to ...$$
\end{proposition}

Finally, for those who are concerned about the coefficient groups for homology, we prove our result with coefficients in $\mathbb{Z}$. The same argument works for arbitrary field coefficients. Homological computations of our numerical simulations are done with coefficients in $\mathbb{Z}/2\mathbb{Z}$.

\subsection{Examples of Simplicial Complexes and Subcomplexes}
\label{sec:topo_notations}

In this section, we define a few simplicial complexes and subcomplexes that will appear in the proof.

\begin{figure}[t]
\centering
\begin{tikzpicture}[x={(2cm,0cm)},y={(0cm,2cm)}, thick, scale=0.5]
  \coordinate (A) at (1,0);
  \coordinate (B) at (0,1);
  \coordinate (C) at (-1,0);
  \coordinate (D) at (0,-1);
  
  \draw (A) node[circle,fill,inner sep=2pt] {};
  \draw (B) node[circle,fill,inner sep=2pt] {};
  \draw (C) node[circle,fill,inner sep=2pt] {};
  \draw (D) node[circle,fill,inner sep=2pt] {};
  
  \draw (A) -- (B) -- (C) -- (D) -- cycle;
\end{tikzpicture}
$\qquad$
\begin{tikzpicture}[x={(2cm,0cm)},y={(0cm,2cm)}, thick, scale=0.5]
  \coordinate (A) at (1,0);
  \coordinate (B) at (0,1);
  \coordinate (C) at (-1,0);
  \coordinate (D) at (0,-1);
  \coordinate (E) at (0,0,0);
  
  \draw (A) node[circle,fill,inner sep=2pt] {};
  \draw (B) node[circle,fill,inner sep=2pt] {};
  \draw (C) node[circle,fill,inner sep=2pt] {};
  \draw (D) node[circle,fill,inner sep=2pt] {};
  \draw (E) node[circle,fill,inner sep=2pt] {};
  
  \draw (A) -- (B) -- (C) -- (D) -- cycle;
  \draw (A) -- (E);
  \draw (B) -- (E);
  \draw (C) -- (E);
  \draw (D) -- (E);
\end{tikzpicture}
\caption{Illustrations of $S^1$ (left) and $D^2$ (right).}
\label{fig:circle_disc}
\end{figure}
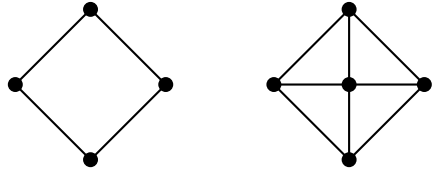

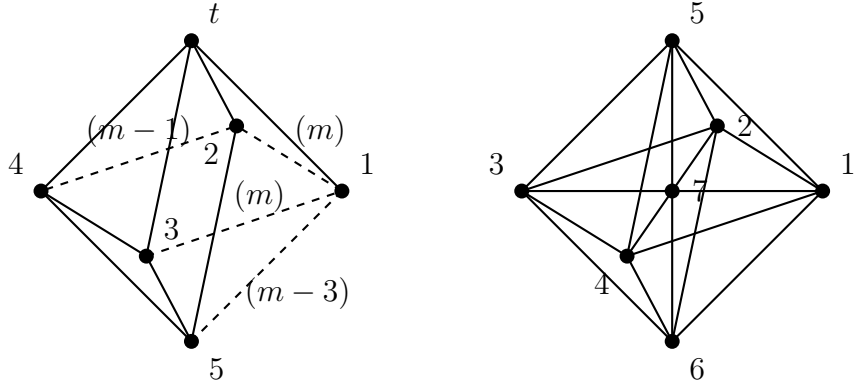
\begin{figure}[t]
\centering
\begin{tikzpicture}[x={(4cm, 0cm)}, y={(1.2cm, 1.732cm)},z={(0cm,4cm)}, thick, scale=0.5]
  \coordinate (A) at (1,0,0);
  \coordinate (B) at (0,1,0);
  \coordinate (C) at (-1,0,0);
  \coordinate (D) at (0,-1,0);
  \coordinate (E) at (0,0,-1);
  \coordinate (F) at (0,0,1);

  \draw (A) node[circle,fill,inner sep=2pt,label=above right:$1$] {};
  \draw (B) node[circle,fill,inner sep=2pt,label=below left:$2$] {};
  \draw (C) node[circle,fill,inner sep=2pt,label=above left:$4$] {};
  \draw (D) node[circle,fill,inner sep=2pt,label=above right:$3$] {};
  \draw (E) node[circle,fill,inner sep=2pt,label=below right:$5$] {};
  \draw (F) node[circle,fill,inner sep=2pt,label=above right:$t$] {};
  
  \draw [dashed] (D) -- node[midway, above]{$\quad (m)$} (A) --node[midway, above]{$\quad \quad (m)$} (B) --node[midway, above]{$(m-1)$} (C);
  \draw (C) -- (D);
  \draw [dashed] (A) --node[midway, below]{$\quad \quad (m-3)$} (E);
  \draw (B) -- (E);
  \draw (C) -- (E);
  \draw (D) -- (E);
  
  \draw (F) -- (A);
  \draw (F) -- (B);
  \draw (F) -- (C);
  \draw (F) -- (D);
  
\end{tikzpicture}
$\qquad$
\begin{tikzpicture}[x={(4cm, 0cm)}, y={(1.2cm, 1.732cm)},z={(0cm,4cm)}, thick, scale=0.5]
  \coordinate (A) at (1,0,0);
  \coordinate (B) at (0,1,0);
  \coordinate (C) at (-1,0,0);
  \coordinate (D) at (0,-1,0);
  \coordinate (E) at (0,0,1);
  \coordinate (F) at (0,0,-1);
  \coordinate (G) at (0,0,0);

  \draw (A) node[circle,fill,inner sep=2pt,label=above right:$1$] {};
  \draw (B) node[circle,fill,inner sep=2pt,label=right:$2$] {};
  \draw (C) node[circle,fill,inner sep=2pt,label=above left:$3$] {};
  \draw (D) node[circle,fill,inner sep=2pt,label=below left:$4$] {};
  \draw (E) node[circle,fill,inner sep=2pt,label=above right:$5$] {};
  \draw (F) node[circle,fill,inner sep=2pt,label=below right:$6$] {};
  \draw (G) node[circle,fill,inner sep=2pt,label=right:$7$] {};
  
  \draw (A) -- (B) -- (C) -- (D) -- cycle;
  \draw (A) -- (E);
  \draw (B) -- (E);
  \draw (C) -- (E);
  \draw (D) -- (E);
  
  \draw (F) -- (A);
  \draw (F) -- (B);
  \draw (F) -- (C);
  \draw (F) -- (D);
  
  \draw (A) -- (G);
  \draw (B) -- (G);
  \draw (C) -- (G);
  \draw (D) -- (G);
  \draw (E) -- (G);
  \draw (F) -- (G);
\end{tikzpicture}
\caption{Illustrations of $S^2$ (left) and $D^3$ (right). The labels and the different line styles for the left illustration are for $\Gamma^{(t)}$ in the proof of \cref{lem:prob_link_contains_sphere}, and those for the right illustration are for \cref{eg:killer}. Labels without parentheses denote node indices in $G(T, \delta, m)$, and labels in parentheses denote edge multiplicity of the dashed edges in $G(T, \delta, m)$.}
\label{fig:sphere_ball}
\end{figure}

We denote by
\begin{itemize}
\item $S^q$ the octahedral $q$-sphere, which has $2(q+1)$ vertices and can be embedded in $\mathbb{R}^{q+1}$ by mapping the vertices to unit vectors on coordinate axes, and
\item $D^{q+1}$ the octahedral $(q+1)$-ball, which is constructed by adding a vertex to $S^q$ and connecting the new vertex with all vertices in $S^q$. It can be embedded in $\mathbb{R}^{q+1}$ by mapping the new vertex to the origin.
\end{itemize}
For instance, $S^1$ and $D^2$, illustrated in \cref{fig:circle_disc}, are the unfilled and filled squares. $S^2$ and $D^3$ are illustrated in \cref{fig:sphere_ball}.
We sometimes abuse notations and denote their underlying graphs by these symbols. Their homology groups are given in \cref{eqn:sphere_betti}.

We also need the notion of \emph{link}. Let $X$ be a simplicial complex.
\begin{itemize}
\item The \emph{star} of a vertex $v$ in $X$, denoted by $\text{St}_X(v)$, is the subcomplex of $X$ consisting of all simplices containing $v$ (and the faces of these simplices). (Our notion of star is called the \emph{closed} star in Definition 3.24 of \citep{giblin_2010_homology} and p.11 of \citep{munkres84algtopo}.)
\item The \emph{link} of a vertex $v$ in $X$, denoted by $\text{Lk}_X(v)$, is the subcomplex of $\text{St}_X(v)$ consisting of all simplices that does \emph{not} contain $v$. (Cf. Definition 3.28 of \citep{giblin_2010_homology} and p.11 of \citep{munkres84algtopo}.)
\end{itemize}

For instance, the link of the central vertex in $D^2$ is $S^1$.

The star of a vertex has the same Betti numbers as $D^{q+1}$, which are stated in \cref{eqn:sphere_betti}, because they are both \emph{contractible}.

\subsection{Minimal Clique Cycles}
\label{sec:topological_facts}

In this section, we state a result, namely \cref{cor:minimallyNontrivialSubcomplex_1skeleton}, about minimal cycles in clique complexes, which will be needed to characterize the $\Gamma_k$'s in the Introduction as the dominating cycles. It is a slight generalization of Lemmas 5.2 and 5.3 in \citep{kahle09_randomCliqueComplex} and Lemma 2.1.4 in \citep{gal05_clique_spheres} to the context of relative homology. Its proof, which is based on an exact-sequence argument, is deferred to \cref{sec:proof_minimallyNontrivialSubcomplex_1skeleton}.

\begin{definition}[Clique-Minimal Complex]\label{def:minimal_complex}
Let $X$ be a clique complex and $A$ be a subcomplex. For $q \geq 0$, $X$ is said to be $(A, q)$-clique-minimal if for every clique subcomplex $Y$ of $X$ that contains $A$, 
$$\beta_q(Y, A) > 0 \text{ if and only if } X = Y.$$ If $A$ is the empty set, we abbreviate $(A, q)$-clique-minimality as $q$-clique-minimality.
\end{definition}

\begin{proposition}\label{cor:minimallyNontrivialSubcomplex_1skeleton}
Let $q \geq 0$ and $A$ be an induced subcomplex of a clique complex $X$ (i.e. $A$ contains a simplex $\sigma$ of $X$ whenever $A$ contains all vertices of $\sigma$). Suppose $X$ is $(A, q)$-clique-minimal. Then
\begin{itemize}
\item if $A$ is just a vertex, then $X$ has at least $2q+2$ vertices; otherwise, $X$ has at least one more vertex than $A$ does; and
\item $\text{deg } v \geq 2q$ for every vertex $v$ in $X$ not in $A$, where $\text{deg}$ denotes the degree of a vertex in the underlying graph of $X$.
\end{itemize}
\end{proposition}

\begin{remark}
Note that a clique-minimal complex is not necessarily minimal (defined by dropping all instances of ``clique" in the definition above). For example, let $A$ be a triangulated annulus whose inner boundary has three edges. Consider the double cover $T$ of $A$, which is formed by gluing together two identical copies of $A$ along the boundaries. Let $T'$ be the space formed by gluing a triangle to the inner boundary of $A$ in $T$. Then $T$ is 2-minimal but not 2-clique minimal. $T'$ is not 2-minimal, but it is 2-clique minimal, if the triangulation of $A$ is nice enough.
\end{remark}

\section{A Decomposition Result}
\label{sec:main_idea}

In this section, we begin our proof of \cref{thm:expected_BettiNumbers} by establishing a decomposition result, \cref{prop:decomposition}, which estimates the Betti numbers of the whole complex in terms of the Betti numbers of its links. After stating the proposition, we describe its relationship with the discussion on dominating cycles in the Introduction. We conclude the section by proving the proposition.

Let $X$ be a clique complex with vertices $\{1, ..., T\}$.

\begin{itemize}
\item Let $X^{(t)}$ be the subcomplex of $X$ such that $X^{(t)}$ consists of all simplices of $X$ whose vertices are all in $\{1, ..., t\}$. 
\item Let $L^{(t)}$ be the link of $t$ in $X^{(t)}$, 
and $f^{(t)}: L^{(t)} \to X^{(t-1)}$ be the inclusion map ($L^{(t)} \subseteq X^{(t-1)}$ because $t$ itself does not lie in its link).
\item \new{For a subcomplex $S$ of $X$, a nonnegative integer $q$, nodes $s, t$ of $X$ such that all indices of nodes in $S$ are smaller than $s$, and $s < t$, let $\mathcal{S}(S, q, s, t)$ be the event where
	\begin{itemize}
	\item $S$ is isomorphic to $S^{q-1}$, and
	\item $S \subseteq L^{(s)} \cap L^{(t)}$.
	\end{itemize}}
\end{itemize}

\new{
Recall that $S^{q-1}$ is the sphere in \cref{sec:topo_notations}.
Note that when $\mathcal{S}(S, q, s, t)$ happens, $X$ contains a $q$-dimensional sphere with $S$ as the equator and $s$ and $t$ as the poles.}

We need the following terms for our estimate.
\begin{align*}
\ell^{(t)}\new{(S, s)} &= \mathbf{1}[\mathcal{S}\new{(S, q, s, t)}]
\\
b^{(t)}_{IK}\new{(S, s)} &= \mathbf{1}[\mathcal{S}\new{(S, q, s, t)}] \mathbf{1}[\beta_q(L^{(t)}, \new{S}) > 0]
\\
b^{(t)}_{KL} &= \beta_q(L^{(t)})
\\
u^{(t)} &= \beta_{q-1}(L^{(t)}),
\end{align*}
where $\mathbf{1}[\Lambda]$ denotes the indicator of the event $\Lambda$, and $f^{(t)}_{q-1}$ is the homomorphism between homology groups at dimension $q-1$ induced by the map $f^{(t)}$.
Technically these terms should be subscripted by $q$, but we drop it since we will not change the dimension $q$.

The letters $u$ and $\ell$ stand for upper and lower bounds, and $b$ stands for boundary. The subscripts $IK$ and $KL$ denote the two types of boundaries mentioned in the Introduction, and we will explain them below. Note that despite their notational difference, $u^{(t)}$ and $b^{(t)}_{KL}$ are just Betti numbers of $L^{(t)}$ of different dimensions. They will be estimated in the same way in \cref{lem:link_BettiNumber}.

We now state our decomposition result.

\begin{proposition}\label{prop:decomposition}
Let $X$ be a clique complex. 
\new{
Let $q \geq 2$, $S$ be a subcomplex of $X$, and $s$ be a node whose label is larger than all node labels in $S$.
}
Then
\begin{itemize}
\item $\sum_{t \leq T} \left[\rk \ker f_{q-1}^{(t)} - \beta_q(L^{(t)})\right] \leq \beta_q(X) \leq \sum_{t \leq T} \beta_{q-1}(L^{(t)})$;
\item \new{
Suppose $\mathcal{S}(S,q,s,t)$ happens. Let $j: S \to L^{(t)}$ be the inclusion map. If $\rk \ker j_{q-1} = 0$, then $\rk \ker f^{(t)}_{q-1} \geq 1$; and
}
\item $\sum_{\new{s < t} \leq T} (\ell^{(t)}\new{(S,s)} - b^{(t)}_{IK}\new{(S, s)}) - \new{\sum_{t \leq T}} b^{(t)}_{KL} \leq \beta_q(X) \leq \sum_{t \leq T} u^{(t)}.$
\end{itemize}
\end{proposition}

\new{The most important statement is the last bullet point. The first bullet point is merely a stepping stone to it. The second bullet point, while also a stepping stone, will also be used in our discussion on simulation results.}

Returning to the discussion of $\Gamma_k$'s in the Introduction, one may check that $u^{(t)} = 1$ on the event where $L^{(t)}$ is a square or a pentagon. In the former case the square along with vertex $t$ gives a copy of $\Gamma_1$. We will show that the count of squares is the dominating term for $u^{(t)}$ in expectation at the end of the proof of \cref{lem:link_BettiNumber}. Similarly, $\ell^{(t)}$ counts a subset of copies of $\Gamma_1$'s in the complex.

$KL$ stands for ``KilL", as a cycle is killed in the homology group when a boundary is formed. We give an example where $b_{KL}^{(t)} = 1$.

\begin{example}[Kill]\label{eg:killer}
Let $q = 2$ and $t = 7$. Let $X^{(6)} \cong S^2$, and $X^{(7)}$ be the clique complex with vertex 7 connected to $1, 2, 3, 4, 5, 6$. The underlying graph of $X^{(7)}$ is shown on the right panel of \cref{fig:sphere_ball}. Then $X^{(7)} \cong D^3$. The Betti number at dimension 2 drops by 1 when the new simplices are added.
\end{example}

$IK$ stands for ``Instant Kill", as the defining event happens when a new cycle is killed soon as it forms. We give an example where $b_{IK}^{(t)} = 1$.

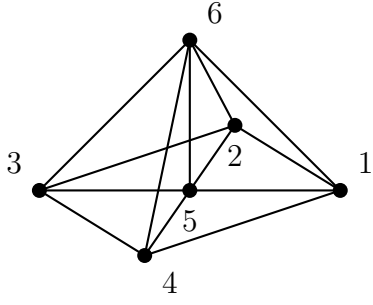
\begin{figure}[t]
\centering
\begin{tikzpicture}[x={(4cm, 0cm)}, y={(1.2cm, 1.732cm)},z={(0cm,4cm)}, thick, scale=0.5]
  \coordinate (A) at (1,0,0);
  \coordinate (B) at (0,1,0);
  \coordinate (C) at (-1,0,0);
  \coordinate (D) at (0,-1,0);
  \coordinate (E) at (0,0,0);
  \coordinate (F) at (0,0,1);

  \draw (A) node[circle,fill,inner sep=2pt,label=above right:$1$] {};
  \draw (B) node[circle,fill,inner sep=2pt,label=below:$2$] {};
  \draw (C) node[circle,fill,inner sep=2pt,label=above left:$3$] {};
  \draw (D) node[circle,fill,inner sep=2pt,label=below right:$4$] {};
  \draw (E) node[circle,fill,inner sep=2pt,label= below:$5$] {};
  \draw (F) node[circle,fill,inner sep=2pt,label=above right:$6$] {};
  
  \draw (A) -- (B) -- (C) -- (D) -- cycle;
  \draw (A) -- (E);
  \draw (B) -- (E);
  \draw (C) -- (E);
  \draw (D) -- (E);
  
  \draw (F) -- (A);
  \draw (F) -- (B);
  \draw (F) -- (C);
  \draw (F) -- (D);
  \draw (F) -- (E);
  
\end{tikzpicture}
\caption{Illustration for the underlying graph of $X^{(6)}$ in \cref{eg:miscarriage}.}
\label{fig:miscarriage}
\end{figure}

\begin{example}[Instant Kill]\label{eg:miscarriage}
Let $q = 2$ and $t = 6$. Let $X^{(5)} \cong D^2$ with vertices $1,2,3,4$ on the boundary and vertex $5$ in the center. \new{Let $S = X^{(4)}$.} Let $X^{(6)}$, which is illustrated in \cref{fig:miscarriage}, be the clique complex with vertex $6$ connected to $1, 2, 3, 4, 5$. The addition of new simplices creates a 2-cycle, namely the sum of the two copies of (the clique complex of) $\Gamma_1$ spanned by $1, 2, 3, 4, 5$ and $1, 2, 3, 4, 6$. However, this does not add to the Betti number, since the new cycle is also the boundary of the sum of all tetrahedra in $X^{(6)}$. One may check that indeed $b_{IK}^{(6)}\new{(X^{(4)},5)} = 1$, as the exact sequence of the pair $(L^{(6)}, X^{(4)})$ shows $H_2(L^{(6)}, X^{(4)}) \cong H_1(X^{(4)}) \cong \mathbb{Z}$.
\end{example}

We now prove \cref{prop:decomposition}.

\begin{proof}[Proof of \cref{prop:decomposition}]

It suffices to show the first two claims. The last claim is a straight-forward corollary.

For the first claim, let $C$ be the star of $t$ in $X^{(t)}$. Then $C$ is contractible. Consider the Mayer-Vietoris sequence for the decomposition $X^{(t)} = X^{(t-1)} \cup C$:
$$H_q(L^{(t)}) \xrightarrow{f^{(t)}_q} H_q(X^{(t-1)}) \to H_q(X^{(t)}) \to H_{q-1}(L^{(t)}) \xrightarrow{f^{(t)}_{q-1}} H_{q-1}(X^{(t-1)}),$$
where the trivial summands $H_*(C) \cong 0$ are suppressed.
Hence we have the short exact sequence
$$
0 \to \im f^{(t)}_q \to H_q(X^{(t-1)}) \to H_q(X^{(t)}) \to 
\ker f^{(t)}_{q-1} \to 0.$$
Since the alternating sum of ranks vanishes, we have
$$\rk  H_q(X^{(t)}) - \rk  H_q(X^{(t-1)}) = \rk \ker f^{(t)}_{q-1} - \rk \im f^{(t)}_q,$$
where $\rk \ker f^{(t)}_{q-1} \leq \rk H_{q-1}(L^{(t)}) = \beta_{q-1}(L^{(t)})$ because the latter group contains the former, and $\rk \im f^{(t)}_{q} \leq \beta_{q}(L^{(t)})$ by the rank-nullity theorem.
The first claim then follows by summing over $t$.

For the second claim, \new{c}onsider the following commutative diagram. 
$$\begin{tikzcd}
... \to H_{q}(L^{(t)}, \new{S}) \ar[r, "\partial"]
& H_{q-1}(\new{S}) \ar[r, "j_{q-1}"] \ar[d, "0"]
& H_{q-1}(L^{(t)}) \ar[d, "f^{(t)}_{q-1}"]
\\
&H_{q-1}(\new{\text{St}_{X^{(t-1)}}(s)}) \ar[r] & H_{q-1}(X^{(t-1)}).
\end{tikzcd}$$

\new{
The maps in the rectangle are all induced by inclusion.  The top row comes from the exact sequence for the pair $(L^{(t)}, \new{S})$. The left vertical map is $0$ because $\text{St}_{X^{(t-1)}}(s)$ is acyclic (Theorem 6.8 of \citep{giblin_2010_homology}, Theorem 8.2 of \citep{munkres84algtopo})}.

\new{
Commutativity of the rectangle then implies $\im j_{q-1}$ lies in $\ker f^{(t)}_{q-1}$, and hence
$$\rk \ker f^{(t)}_{q-1} \geq \rk \im j_{q-1} = 1 - \rk \ker j_{q-1} = 1,$$
where the first equality follows from rank-nullity theorem (recall $\beta_{q-1}(S) = 1$), and the second one holds by assumption. The second bullet point then follows.


For the last claim,
exactness of the top row at $H_{q-1}(S)$ gives
\begin{equation}\label{bound_relative_Betti_num}
\rk \ker j_{q-1} = \rk \im \partial \leq \rk H_q(L^{(t)}, S).
\end{equation}
Hence,
\begin{align*}
\rk \ker f_{q-1}^{(t)} &\geq \rk \ker f_{q-1}^{(t)} \mathbf{1}[\mathcal{S}(S,q,s,t)]\mathbf{1}[\rk \ker j_{q-1} = 0]
\\&\geq \mathbf{1}[\mathcal{S}(S,q,s,t)]\mathbf{1}[\rk \ker j_{q-1} = 0] & \text{(by second claim)}
\\&= \mathbf{1}[\mathcal{S}(S,q,s,t)](1-\mathbf{1}[\rk \ker j_{q-1} > 0])
\\&\geq \mathbf{1}[\mathcal{S}(S,q,s,t)](1-\mathbf{1}[\beta_q(L^{(t)},S) > 0]) & \text{(by \cref{bound_relative_Betti_num})}.
\end{align*}
The last claim then follows from applying the first claim and plugging in definitions.}
\end{proof}

\section{Proof of \cref{thm:expected_BettiNumbers}}
\label{sec:proof_main_result}

It remains to estimate the expectations of all the terms in \cref{prop:decomposition} for the preferential attachment complex \new{for some choice of $S$ and $s$. Throughout our proof, we fix $S = X^{(2q)}$ and $s = 2q+1$.}

A direct application of \cref{prop:subgraphsProbs_preferentialAttachment} gives a lower bound on $\sum E[\ell^{(t)}\new{(X^{(2q)}, 2q+1)}]$.

\begin{lemma}\label{lem:prob_link_contains_sphere}
Consider the preferential attachment complex $X = X(T, \delta, m)$. Let $q \geq 0$ and suppose $m \geq 2q$. Then
$$\sum_{t \leq T} E[\ell^{(t)}\new{(X^{(2q)}, 2q+1)}] = 
\begin{cases}
\Omega(T^{1 - 2q\chi(\delta, m)}) & \text{ if } 1 - 2q\chi(\delta, m) > 0 \\
\Omega(\log T) & \text{ if } 1 - 2q\chi(\delta, m) = 0
\end{cases}
$$
where the big-$\Omega$ constants depend only on $q, \delta, m$.
\end{lemma}

\begin{proof}
For each $t > 2q + 1$, let $\Gamma^{(t)}$ be a subgraph of $U(t, m)$ (possibly with repeated edges) with the following properties:
\begin{itemize}
\item the vertices are $1$, ..., $2q + 1$ and $t$;
\item each of $t$ and $2q +1$ is connected to each of $1$, ..., $2q$, and $t$ is not connected to $2q+1$;
\item all edges incident on $t$ are simple;
\item the out-degree of every vertex other than $1$ and $t$ is $m$; and
\item removing $t$ and edges incident on $t$ and replacing repeated edges with simple edges gives the underlying graph of $D^{q}$.
\end{itemize}
For example, for $q = 2$, $\Gamma^{(t)}$, can be the graph illustrated on the left panel of \cref{fig:sphere_ball}, which by definition has the following edges:
\begin{itemize}
\item $m$ edges between $1$ and $2$ and $m$ edges between $1$ and $3$,
\item $m-1$ edges between $4$ and $2$ and 1 edge between $4$ and $3$,
\item $m-3$ edges between $5$ and $1$ and 1 edge between $5$ and each of $2, 3, 4$, and
\item 1 edge between $t$ and each of $1, 2, 3, 4$.
\end{itemize}

The $\Gamma^{(t)}$'s can be chosen to be isomorphic to each other. Then \new{$\mathcal{S}(X^{(2q)}, q, 2q+1, t)$} holds for $X(T, \delta, m)$ whenever $X(T, \delta, m)$ contains $\Gamma^{(t)}$.

To simplify notations, let \new{$p(v) = p_{\Gamma^{(t)}}(v)$ for every vertex $v$ of $\Gamma^{(t)}$}. Note that $p(t) = -2q\chi(\delta, m).$ By \cref{prop:subgraphsProbs_preferentialAttachment}, we have
$$P[\new{\mathcal{S}(X^{(2q)}, q, 2q+1, t)}] = \Omega(t^{p(t)} \prod_{k \leq 2q+1} k^{p(k)} ) = \Omega (t^{p(t)}) = \Omega(t^{-2q\chi(\delta, m)}).$$
Summing over $t$ and applying integral test gives the desired result.
\end{proof}

Next we use \cref{thm:subgraphsCounts_preferentialAttachment,cor:minimallyNontrivialSubcomplex_1skeleton} to estimate $\sum E[u^{(t)}], \sum E[b_{KL}^{(t)}]$ and $\sum E[b^{(t)}_{IK}\new{(X^{(2q)}, 2q+1)}]$, but before that, we need an auxiliary lemma to simplify the application of \cref{thm:subgraphsCounts_preferentialAttachment}.

\begin{lemma}\label{lem:PAM_count_sequence_increment}
Let $\Gamma$ be a subgraph of $U(T, m)$ with vertex set $V_\Gamma = \{v_1 < ... < v_{|V_H|}\}$. Let $d_k$ and $\indeg_k$ be the total degree (sum of in- and out-degrees) and the in-degree of node $v_k$ with respect to $\Gamma$.
Then the sequence $(a_k)$ in \cref{thm:subgraphsCounts_preferentialAttachment} satisfies
$$a_k - a_{k-1}
\begin{cases}
= d_k\chi(\delta, m) - 1 & \text{ if } \indeg_k = 0\\
\geq (d_k - 2) \chi(\delta, m) & \text{ if } \indeg_k > 0.
\end{cases}$$ 
\end{lemma}

\begin{proof}
Direct verification.
\end{proof}

The next lemma gives a matching upper bound on $\sum E[u^{(t)}]$ and an upper bound on $\sum E[b_{DR}^{(t)}]$ with a smaller order of magnitude.

\begin{lemma}\label{lem:link_BettiNumber}
Consider the preferential attachment complex $X = X(T, \delta, m)$. Let $q \geq 1$ and suppose $m \geq 2(q + 1)$. Then
$$\sum_{t \leq T} E[\beta_{q}(L^{(t)})] = \begin{cases}
O(T^{1 - 2(q+1)\chi(\delta, m)}) & \text{ if } 1 - 2(q+1)\chi(\delta, m) > 0 \\
O(\log T) & \text{ if } 1 - 2(q+1)\chi(\delta, m) = 0 \\
O(1) & \text{ otherwise,}
\end{cases}$$
where the big-Oh constants depend only on $q, \delta, m$.
\end{lemma}

\begin{proof}

Since $L^{(t)}$ has at most $m$ vertices, it has at most ${m \choose {q+1}}$ $q$-dimensional simplices, where ${m \choose {q+1}}$ denotes the binomial coefficient. By the weak Morse inequality (Theorem 1.7 of \citep{forman02_discreteMorse_guide}), 
\begin{equation}\label{eqn:betti_link_estimate}
\beta_q(L^{(t)}) \leq {m \choose {q+1}} \mathbf{1}[\beta_q(L^{(t)}) > 0],
\end{equation}
and hence the $\sum \beta_q(L^{(t)})$ is at most ${m \choose {q+1}}$ times the number of $t$'s such that $\beta_q(L^{(t)}) > 0$. We will construct a distinct graph $\Gamma^{(t)}$ for each such $t$, and bound the expected count of such graphs.

Whenever $\beta_q(L^{(t)}) > 0$, $L^{(t)}$ contains a $q$-clique-minimal subcomplex (note that $L^{(t)}$ is also clique, see \cref{lem:clique_closure}), which by \cref{cor:minimallyNontrivialSubcomplex_1skeleton}, has at least $2q + 2$ vertices, whose degrees are at least $2q$. Since these vertices are all connected to node $t$ in $G(T, \delta, m)$, this gives rise to a subgraph $\Gamma^{(t)}$ in $G(T, \delta, m)$ with the following properties:
\begin{itemize}
\item $\Gamma^{(t)}$ has at least $2q + 3$ vertices, and at most $m + 1$ vertices,
\item the vertices of $\Gamma^{(t)}$ all have degrees at least $2q + 1$, and
\item the last vertex of $\Gamma^{(t)}$ (which is $t$) is connected to all other vertices.
\end{itemize}

Note that for $t \neq s$, $\Gamma^{(t)}$ and $\Gamma^{(s)}$ are distinct subgraphs in $G(T, \delta, m)$ because their last nodes are different.

Therefore, \cref{eqn:betti_link_estimate} implies $\sum E[\beta_q(L^{(t)})]$ is at most ${m \choose {q+1}}$ times the expected number of subgraphs of $G(T, \delta, m)$ satisfying the properties above.

We use \cref{thm:subgraphsCounts_preferentialAttachment} to give an upper bound on the expected count of all such subgraphs. Since such subgraphs have at most $m+1$ vertices, there are only finitely many isomorphism classes of such graphs. Fix an isomorphism class and pick a representative $\Gamma$ in the class.

We claim the sequence $(a_k)$ in \cref{eqn:PAM_count_sequence} attains its maximum at $|V_\Gamma|-1$ or $|V_\Gamma|$. To establish this claim, it suffices to show $a_0 < a_1 < ... < a_{|V_\Gamma|-1}$. Let $0 < k < |V_\Gamma|$. Then $v_k$ is not the last vertex of $\Gamma$. Since its degree is at least $2q + 1 \geq 3$, and since it is connected to the last vertex, \cref{lem:PAM_count_sequence_increment} implies that 
$$a_k - a_{k-1} \geq (3-2) \chi(\delta, m) > 0.$$
The claim then follows.

By definition, $a_{|V_\Gamma|} = 0$. Hence, by \cref{lem:PAM_count_sequence_increment} again, 
$$a_{|V_\Gamma| - 1} = -(a_{|V_\Gamma|} - a_{|V_\Gamma| - 1}) = 1 - d_\Gamma \chi(\delta, m),$$ where $d_\Gamma$ is the degree of the last node with respect to $\Gamma$. Therefore, the expected count of $\Gamma$ is
$$O(T^{A_\Gamma} \log^{r_\Gamma} T),$$
where
\begin{align*}
A_\Gamma &= \max\left(0, 1 - d_\Gamma \chi(\delta, m) \right)\\
r_\Gamma &= {\mathbf{1}[1 - d_\Gamma \chi(\delta, m) = 0]}.
\end{align*}
The sum of counts for all isomorphism classes is dominated by the classes of the $\Gamma$'s with the minimum $d_\Gamma$. Our criteria for $\Gamma$ require $d_\Gamma \geq 2q + 2$. The result then follows. Note that the minimal $d_\Gamma$ is attained by $D^{q+1}$, which is the only minimizer such that the corresponding $\beta_q(L^{(t)})$ is positive, and this justifies our discussion of $\Gamma_1$'s in the introduction.
\end{proof}

A similar argument using relative homology gives an upper bound on $\sum_{t \leq T} E[b_{IK}^{(t)}\new{(X^{(2q)}, 2q+1)}]$ with a smaller order of magnitude.

\begin{lemma}\label{lem:miscarriage}
Suppose $q \geq 2$ and $m \geq 2q$. \new{Let $S$ be a (possibly random) subcomplex $S$ of $X(T, \delta, m)$, and $s$ be a (possibly random) node in $X(T, \delta, m)$ that is (almost surely) a later node than all nodes in $S$.} Then
\begin{equation*}
\sum_{t \leq T} E[b_{IK}^{(t)}\new{(S, s)}] = 
\begin{cases}
O(T^{1 - (2q+1)\chi(\delta, m)}) & \text{ if } 1 - (2q+1)\chi(\delta, m) > 0 \\
O(\log T) & \text{ if } 1 - (2q+1)\chi(\delta, m) = 0 \\
O(1) & \text{ otherwise,}
\end{cases}
\end{equation*}
where the big-Oh constants depend only on $q, \delta, m$.
\end{lemma}

\begin{proof}
Again, on the event $\beta_q(L^{(t)}, X^{(2q)}) > 0$, $L^{(t)}$ contains an $(\new{S}, q)$-clique-minimal subcomplex, which by \cref{cor:minimallyNontrivialSubcomplex_1skeleton}, has at least $2q + 1$ vertices, $2q$ of them (from \new{$S$}) with degree at least $(2q-2)$, and the rest with degree at least $2q$.

Since these vertices are all connected to node $t$ in $G(T, \delta, m)$, this gives rise to a subgraph $\Gamma^{(t)}$ in $G(T, \delta, m)$ with the following properties:
\begin{itemize}
\item $\Gamma^{(t)}$ has at least $2q + 2$ vertices, and at most $m + 1$ vertices;
\item $2q$ vertices of $\Gamma$ have degrees at least $2q-1$, and the rest have degrees at least $2q+1$; and
\item the last vertex of $\Gamma^{(t)}$ (which is $t$) is connected to all other vertices.
\end{itemize}

Then $E[b_{IK}^{(t)}\new{(S,s)}]$ is at most the expected number of such subgraphs.

Appealing to \cref{thm:subgraphsCounts_preferentialAttachment} again, for each such $\Gamma$, the maximum $A_\Gamma$ of the sequence $(a_k)$ is attained by one of the last two terms. The result then follows.
\end{proof}

The proof of \cref{thm:expected_BettiNumbers} is now complete by plugging in estimates in \cref{lem:prob_link_contains_sphere,lem:link_BettiNumber,lem:miscarriage} to \cref{prop:decomposition}.

\section{Proof of \cref{prop:trivial_cases}}
\label{sec:proof_trivial_cases}

The claim for $q = 0$ is trivial, because preferential attachment graphs are connected by construction. The claim for $m < 2q$ follows from the fact that there are not enough edges to form $q$-dimensional holes. This can be seen by applying \cref{cor:minimallyNontrivialSubcomplex_1skeleton} to the last node in a hypothetical $q$-minimal subcomplex.

It remains to prove the case for $q = 1$ with Morse inequality. Let $|\bar V|, |\bar E|, |\bar F|$ be the expected numbers of vertices, edges and triangles in $X(T, \delta, m)$.
The strong Morse inequality (Theorem 1.8 of \citep{forman02_discreteMorse_guide}) implies that
$$|\bar E| - |\bar V| - |\bar F| \leq E \beta_1(X(T, \delta, m)) \leq E \beta_0(X(T, \delta, m)) + |\bar E| - |\bar V|.$$
Obviously, $|\bar V| = T$ and $\beta_0(X(T, \delta, m)) = 1$. \cref{thm:subgraphsCounts_preferentialAttachment} implies the expected numbers of triangles and of bi-angles (the two-node graph with two distinct edges from one node to the other) are both $o(T)$, and hence 
\begin{gather*}
m(T-1) - 2o(T) \leq |\bar E| \leq m(T-1)\\
|\bar F| = o(T).
\end{gather*}
The result then follows.

\section{Proof of \cref{cor:minimallyNontrivialSubcomplex_1skeleton}}
\label{sec:proof_minimallyNontrivialSubcomplex_1skeleton}

Our proof follows the argument of Lemmas 5.2 and 5.3 in \citep{kahle09_randomCliqueComplex}. We generalize these lemmas to the setting of relative homology in \cref{prop:kahle_minimal_sphere_clique_complexes}. To ensure clique-minimality conditions are met in our argument, we need \cref{lem:clique_closure} to ensure certain subcomplexes of a clique complex are clique complexes. \cref{cor:minimallyNontrivialSubcomplex_1skeleton} is a corollary of \cref{prop:kahle_minimal_sphere_clique_complexes}.

For every simplicial complex $X$ and every vertex $v$ of $X$, we denote by $X - v$ the simplicial complex that consists precisely of simplices that do not contain $v$.

\begin{lemma}\label{lem:clique_closure}
If $X$ is a clique complex, then 
$\text{Lk}_X v$ and $X - v$ are clique complexes for every vertex $v$ in $X$.

\end{lemma}

\begin{proof}



The claim for $X - v$ is trivial. For $\text{Lk}_X v$, let $w_0, ..., w_q$ be distinct vertices in $\text{Lk}_X v$ and $\sigma$ be the simplex whose vertex set is $\{w_0, ..., w_q\}$.
Since each $w_i$ lies in the link, it is a vertex of a simplex in $\text{St}_X v$, and hence $X$ contains the edge from $v$ to the $w_i$ as well. Since this is true for all $w_i$'s, $X$ contains the simplex $\sigma_v$ whose vertex set is $\{w_0, ..., w_q\} \cup \{v\}$. By definition $\sigma_v$ lies in the star, and hence $\sigma$ does too. Since $\sigma$ does not contain $v$, $\sigma$ lies in the link.
\end{proof}

Next, we generalize Lemmas 5.2 and 5.3 of \citep{kahle09_randomCliqueComplex} to the setting of relative homology. Only the first part of following lemma is novel. The second claim is Lemma 2.1.4 of \citep{gal05_clique_spheres} and Lemma 5.3 of \citep{kahle09_randomCliqueComplex} phrased differently. We reproduce their proofs with our terminologies.

\begin{lemma}\label{prop:kahle_minimal_sphere_clique_complexes}
Let $X$ be a clique complex and $A$ be a (not necessarily clique) subcomplex of $X$. Suppose $X$ is $(A, q)$-clique-minimal. Then following statements are true.
\begin{enumerate}
\item If $q > 0$ then,
$$\beta_{q-1}(\text{Lk}_{X}(v), B) > 0$$
for every vertex $v$ in $X$ but not in $A$ and every (possibly empty) acyclic subcomplex $B$ of $\text{Lk}_{X}(v) \cap A$. (Acyclic means $H_0$ is $\mathbb{Z}$ or $0$ and all other homology groups are $0$.)
\item If $q \geq 0$ and $A$ consists of one single vertex, then $X$ has at least $2q + 2$ vertices.
\end{enumerate}
\end{lemma}

\begin{proof}
For the first claim, fix a vertex $v$ in $X$ but not in $A$. We have the following commutative diagram.
$$\begin{tikzcd}
H_q(\text{Lk}_X v, B) \ar[r] \ar[d]
& H_q(\text{St}_X v, B) \ar[r, "0"] \ar[d]
& H_q (\text{St}_X v, \text{Lk}_X v) \ar[r, hook, "\psi"] \ar[d, "\varphi; EX"]
& H_{q-1}(\text{Lk}_X v, B) \ar[d]
\\
H_q(X - v, A) \ar[r, "\text{rk } 0"]
& H_q(X, A) \ar[r]
& H_q(X, X - v) \ar[r] 
& H_{q-1}(X - v, A),
\end{tikzcd}$$
where the two rows are long exact sequences of triplets, and the vertical maps are induced by inclusion. We would like to show the top-right group has positive rank.

We explain the annotations in the diagram. The map $\varphi$, marked by ``EX'', is an isomorphism by excision theorem (Theorem 6.4 of \citep{giblin_2010_homology}, Theorem 9.1 of \citep{munkres84algtopo} and Theorem 2.20 of \citep{hatcher02_algtopo}). The map marked by ``0" is zero because $\text{St}_X v$ is contractible (and $q \geq 1$), and map marked by ``rk 0" is rank-0 by the clique-minimality of $X$. Exactness then implies $\psi$ is injective and $\beta_q(X, X - v) \geq \beta_q(X, A) > 0$, where the last inequality holds by assumption.

Since the top-right group contains $\psi\varphi^{-1} H_q(X, X - v)$, it must have positive rank. The first claim then follows.

For the second claim, the case for $q = 0$ is trivial. For $q > 0$, suppose for contradiction that $X$ has strictly fewer than $2q + 2$ vertices.

We first consider the main case when there is a vertex $v$ connected to the vertex in $A$. The first claim implies $\beta_{q-1}(\text{Lk}_X(v), A) > 0$. Therefore, $\text{Lk}_X(v)$ has an $(A, q-1)$-clique-minimal subcomplex, and hence by induction, the link has at least $2q$ nodes. Since we have assumed $X$ has strictly fewer than $2q + 2$ vertices, all nodes other than $v$ are in the link. In other words, $X$ is $\text{St}_X(v)$, and hence is contractible, contradictory to the assumption of $(A, q)$-clique minimality. \new{[Removed redundant paragraph.]}

We now consider the case when $A$ is not connected to any other vertices. Since $\beta_q(X, A) > 0$, $X - A$ has at least one edge, say with endpoints $v, w$ in $X - A$. It can be directly verified that $X - A$ is $q$-minimal, and hence $(\{w\}, q)$-minimal. The main case above implies $X-A$ has at least $2q+2$ vertices, and hence $X$ does too.
\end{proof}

\begin{proof}[Proof of \cref{cor:minimallyNontrivialSubcomplex_1skeleton}]
The first part of the first claim is just the second claim of  \cref{prop:kahle_minimal_sphere_clique_complexes}. The second part of the first claim is trivial, because if $X$ has the same vertex set as $A$, then $X = A$ (because $A$ is an induced subcomplex), and hence $H_q(X, A) = 0$.

The second claim is trivial for $q = 0$.
For $q > 0$, since the removal of isolated vertices not in $A$ does not change the (relative) Betti number at dimension $q$, clique-minimality implies $\deg v \geq 1$, and hence $\text{Lk}_X(v)$ is nonempty. The first claim in \cref{prop:kahle_minimal_sphere_clique_complexes} implies $\beta_{q-1}(\text{Lk}_X(v), \emptyset) > 0$, and hence $\beta_{q-1}(\text{Lk}_X(v), \{w\}) > 0$ whenever $w \in \text{Lk}_X(v)$. The second claim then implies $\text{Lk}_X(v)$ contains a subcomplex with at least $2q$ vertices. This means that $v$ is connected to at least $2q$ vertices in $X$.
\end{proof}

\section{Numerical Simulation}
\label{sec:numerical_simulation}

We discuss the simulation we mentioned in the Introduction in greater detail. Recall the right panel of \cref{fig:teaser} illustrates the evolution of the mean Betti numbers. Below, we explain the setup of the simulation, and we return to discuss the results shown in \cref{fig:teaser} at the last paragraph.

Numerical computations related to topology and graph theory are done with Ripser \citep{bauer21_ripser,tralie18_ripser} and Igraph \citep{igraph} respectively. Other numerical computations are done with Numpy \citep{numpy} and Scipy \citep{scipy}. Codes are compiled with Numba \citep{lam15_numba}. Plots are generated with Matplotlib \citep{matplotlib}.

500 preferential attachment clique complexes with $T = 10^4$ nodes and with parameters $m = 7$ and $\delta = -5$ are generated. We compute the sample mean of their Betti numbers (with coefficients in $\mathbb{Z}/2\mathbb{Z}$) at dimension $q = 2$.

The black curve corresponds to the evolution of mean Betti numbers. 
We remark that the median of means gives a similar estimate of the expectation.

We also compute the sample mean of the upper bound $\sum u^{(t)}$ in \cref{prop:decomposition} and the sample mean of a lower bound $\sum_{s < t \leq T} (\ell^{(t)}\new{(S, s)} - \hat b_{IK}^{(t)}\new{(S, s)}) - \sum_{t \leq T}b_{KL}^{(t)}$, where \new{$S$, $s$ and $\hat b_{IK}^{(t)}(S, s)$ will} be defined below.
The evolution of the means of these bounds are plotted in dotted lines.
We compute these quantities for graphs with $T = 10^5$ nodes, because their computation is cheaper than that of Betti numbers.

\new{
We now defined $S$, $s$ and $\hat b_{IK}(S, s)$.
\begin{itemize}
\item $S$ is the first induced subcomplex of $X^{(20)}$ (in an arbitrary but deterministically consistent ordering) that is isomorphic to $S^{q-1}$, if one exists; otherwise, it is the first subcomplex of $X^{(20)}$;
\item $s$ is the first node whose label is larger than all of those in $S$ such that $S \subseteq \text{Lk}_X(s)$, if it exists; otherwise it is node 21;
\item $\hat{b}_{IK}^{(t)}(S,s) = \mathbf{1}[\mathcal{S}(S, q,  s, t)] \mathbf{1}[\rk \ker j_{q-1} = 1]$, where $j: S \to L^{(t)}$ denotes the inclusion map.
\end{itemize}

We remark on these definitions:
\begin{itemize}
\item The ``otherwise'' statements in the above definition are unimportant, because in those cases, $\ell^{(t)}\new{(S, s)} = \hat b_{IK}^{(t)}(S, s) = 0$.
\item We do not fix $S = X^{(2q)}$ and $s = 2q+1$ as in our proofs, because $\mathcal{S}(X^{(2q)}, q, 2q+1, t)$ happens too rarely that $E[\sum \ell^{(t)}]$ is too small to be numerically estimated.
\item We change the definition of $\hat{b}_{IK}^{(t)}$ because the computation of relative Betti numbers is numerically inconvenient. We numerically compute $\rk \ker j_{q-1}$ by computing the persistence diagram for the inclusion $S^{(t)} \subseteq L^{(t)}$ with Ripser \citep{bauer21_ripser,tralie18_ripser}. By the first two bullet points of \cref{prop:decomposition}, the new expression does give a lower bound.
\end{itemize}
}

Finally, we draw a band that contains all curves. The slope of the band is determined by \cref{thm:expected_BettiNumbers}. While the discussion beneath the theorem suggests some values of the $y$-intercepts for the band, the corresponding band trivially covers the entire plot. Instead, we manually choose other values of the $y$-intercepts.

It is apparent from the plot that the convergence is slow. In particular, at $T = 10^5$, the mean upper bound still grows at a rate faster than the asymptotic rate. However, it is obvous that the curve is concave, and hence has a decreasing slope. We also note that the mean upper bound is a good approximation of the mean Betti numbers.

\section{Future Directions}
\label{sec:future_directions}

We have established analytically the asymptotics of the expected Betti numbers of the affine preferential attachment clique complexes and illustrated them numerically. A number of open questions remain.

It would be desirable to have sharper estimates of the expected Betti numbers, and finer descriptions of the distributions of the Betti numbers.

Other topological properties of the preferential attachment graphs are also of interest. To understand the robustness of the complex, it would be helpful to understand the evolution of Betti numbers as nodes are removed. One may also consider the Betti numbers of the Rips complexes of the graph with respect to the graph metric. Beyond Betti numbers, one may also consider the homotopy type of the random simplicial complexes. Since holes are filled by later nodes, it is possible that all holes are filled if nodes are added \emph{ad infinitum} and the number of edges added each time grows slowly. In particular, Weinberger conjectured that the resultant complex is contractible in a private conversation.

Other scale-free simplicial complexes are also of interest.

\bibliography{bib_randCpx.bib}
\end{document}